\documentclass[oneside]{amsart}
\usepackage{graphicx} 
\usepackage{float}
\usepackage{amsfonts}
\usepackage{tikz}
\usepackage{bm}
\usepackage {amssymb}
\usepackage {amsmath}
\usepackage {amsthm}
\usepackage{graphicx}
\usepackage{multirow}
\usepackage{xypic}
\usepackage {amscd}
\usepackage{mathrsfs}
\usepackage[colorlinks, linkcolor=blue, anchorcolor=black, citecolor=red]{hyperref}
\usepackage{enumerate}
\usepackage{cleveref}
\usepackage{comment}

\usepackage{geometry}  

\newcommand{\bC}{\mathbb{C}}
\newcommand{\bQ}{\mathbb{Q}}
\newcommand{\bZ}{\mathbb{Z}}
\newcommand{\bR}{\mathbb{R}}
\newcommand{\bD}{\mathbb{D}}

\newcommand{\la}{\langle}
\newcommand{\ra}{\rangle}

\newcommand{\bP}{\mathbb{P}}
\newcommand{\cY}{\mathcal{Y}}
\newcommand{\cO}{\mathcal{O}}

\newcommand{\cM}{\mathcal{M}}
\newcommand{\moduli}{\overline{\mathcal{M}}}

\newcommand{\cF}{\mathcal{F}}
\newcommand{\cE}{\mathcal{E}}
\newcommand{\cZ}{\mathcal{Z}}

\newcommand{\cC}{\mathcal{C}}

\newcommand{\lSFT}{\mathrm{lSFT}}
\newcommand{\CL}{\mathcal{L}}

\newcommand{\LCZ}{\mathrm{LCZ}}

\newcommand{\rd}{\mathrm{d}}
\newcommand{\Id}{\mathrm{id}}

\newcommand{\frg}{{\mathfrak{g}}}
\newcommand{\age}{{\mathrm{age}}}

\newcommand{\mld}{\mathrm{mld}}
\newcommand{\vdim}{\mathrm{vdim}}

\newcommand{\Mor}{\mathrm{Mor}}
\newcommand{\Aut}{\mathrm{Aut}}
\newcommand{\SFT}{\mathrm{SFT}}
\newcommand{\ev}{\mathrm{ev}}

\newtheorem{thm}{Theorem}[section]
\newtheorem{prop}[thm]{Proposition}
\newtheorem{example}[thm]{Example}

\newtheorem{cor}[thm]{Corollary}
\newtheorem{rem}[thm]{Remark}
\newtheorem{conj}[thm]{Conjecture}

\newtheorem{exmp}[thm]{Example}
\newtheorem{lem}[thm]{Lemma}

\newtheorem{defn-prop}[thm]{Definition-Proposition}

\newcommand{\Addresses}{{
		\bigskip
		\footnotesize

  	Chi Li, \par\nopagebreak
        \textsc{Department of Mathematics, Rutgers University, Piscataway, NJ, U.S., 08854-8019.}\par\nopagebreak
         \textit{E-mail address:} \href{mailto:chi.li@rutgers.edu}{chi.li@rutgers.edu}

         \medskip

	    Zhengyi Zhou, \par\nopagebreak
	    \textsc{Morningside Center of Mathematics, Chinese Academy of Sciences;}\par\nopagebreak
         \textsc{Academy of Mathematics and Systems Science, Chinese Academy of Sciences, China, 100190}\par\nopagebreak
		\textit{E-mail address}: \href{mailto:zhyzhou@amss.ac.cn}{zhyzhou@amss.ac.cn}

}}

\begin{document}

\title{Minimal log discrepancy and orbifold curves}
\author{Chi Li, Zhengyi Zhou }

\dedicatory{Dedicated to Professor Gang Tian for his 65th birthday}

\date{}

\begin{abstract}
We show that the minimal log discrepancy of any isolated Fano cone singularity is at most the dimension of the variety. This is based on its relation with dimensions of moduli spaces of orbifold rational curves. We also propose a conjectural characterization of weighted projective spaces as Fano orbifolds in terms of orbifold rational curves, which would imply the equality holds only for smooth points. 
\end{abstract}

\maketitle


\section{Introduction and main result}

In birational algebraic geometry, in particular, in the study of Minimal Model Program (MMP), the minimal log discrepancy (mld) is an important invariant for Kawamata log terminal (klt) singularities (see \cite{Amb06,KM98}). Indeed, Shokurov showed that two conjectural properties of mld would imply the termination of flips along MMP. One of the conjectures says that the mld of closed points is lower semicontinuous when the point moves on a fixed normal variety X of dimension $n$. Since the mld of smooth point is equal to $n$, this in particular implies that $\mld(x, X)\le n$ (see \cite[Conjecture 3.2]{Amb06}). In this short essay, we will prove that this sharp upper bound is indeed true for any isolated Fano cone singularity.
\begin{thm}\label{thm-main}
Let $o\in X$ be an isolated Fano cone singularity of dimension $n$. Then $\mld(o, X)\le n$. 
\end{thm}
\begin{rem}
By \cite[Theorem 1.1]{McL16}, we find that more generally if the link of an isolated singularity is contactomorphic to the link of an isolated Fano cone singularity, then the same inequality is also true. 
\end{rem}

Note that when the Fano cone is an ordinary cone over a smooth Fano manifold, the upper bound is equivalent to the well-known fact that the Fano index of a Fano manifold is at most the dimension plus one. This fact can be proved by using either the Riemannn-Roch theorem or Mori's bend-and-break theory of rational curves. Since the Riemann-Roch approach does not seem to work well in the orbifold setting, our proof uses Mori's theory in the orbifold setting and shows its connection with minimal log discrepancy invariants. This connection is motivated by our previous work \cite{LZ24} which in particular expresses the mld invariant in terms of certain symplectic invariants that arises essentially in the study of moduli space of pseudo-holomorphic curves. 
In this paper we will show that the mld invariant is bounded from above by dimensions of certain moduli space of orbifold rational curves with domain a weighted projective line (see \eqref{ineq-mld}). We will give two proofs for this crucial inequality, one using purely algebraic geometry and the other from a symplectic perspective. 
The argument for the above result also leads us to propose a characterization of weight projective spaces by using orbifold rational curves in the same spirit of Mori-Mukai (see Conjecture \ref{conj-WP}), which would imply that the equality case only happens for smooth points. As evidence, we prove an orbifold version of Mori's theorem:

\begin{thm}\label{thm-main2}
    Let $\cY$ be a complex orbifold whose orbifold tangent bundle $T \cY$ is ample. Then $\cY$ is a finite quotient of a weighted projective space. 
\end{thm}
See \cite{Cheng17} for a related characterization of smooth projective spaces among normal varieties with quotient singularities. 
The method of proof is the same as Mori, which uses a family of orbifold rational curves of minimal degree that pass through a fixed (orbifold) point to sweep out the whole orbifold. The ampleness is used to deduce that there are no obstructions to the deformation theory of such curves. It is important for this purpose that we use orbifold rational curves with only one non-trivial orbifold point, which guarantees the splitting of any orbifold holomorphic vector bundle into a direct sum of orbifold line bundles (see \eqref{eq-split}).

\subsection*{Acknowledgments}
C. Li is partially supported by NSF (Grant No. DMS-2305296). Z. Zhou is supported by National Key R\&D Program of China under Grant No.2023YFA1010500, NSFC-12288201 and NSFC-12231010. C. Li thanks Qile Chen for helpful discussions on twisted stable maps and thanks Chris Woodward for telling him the reference \cite{MT12}. 

\section{Proof of Main results }\label{s2}

We will use the notion in our previous paper \cite{LZ24}. 
The starting point of our proof is a formula for the mld of an isolated Fano cone singularity $(X,o)$ derived in \cite[Proposition 2.12]{LZ24}. To state the formula, we assume that the $X$ is the given as the orbifold cone $C(\cY, \CL)$ where the orbifold base $\cY=(Y, \Delta)$ is obtained as the quotient $(X\setminus o)/\bC^*$ and $\CL$ is the associated orbifold line bundle. The klt condition is equivalent to the condition that $(Y, \Delta)$ is a klt Fano orbifold and $-K_\cY=r \CL$ for some $r\in \bQ_{>0}$. Throughout this paper, we use $(\cY, \CL)$ to denote a Delige-Mumford stack (or equivalently an orbifold) equipped with an orbifold line bundle and use $(Y, L)$ to denote the associated coarse moduli space as an algebraic variety equipped with a $\bQ$-line bundle. 

let $\mu: X'\rightarrow X$ be the extraction of orbifold base $Y$ such that $X'$ is isomorphic to the total space of the orbifold line bundle $\CL^{-1}\rightarrow \cY$. Now $X'$ as an affine variety has only cyclic quotient singularities along zero section $Y_0$. Fix any $p\in Y_0$ and choose a neighborhood $U$ of $p$ such that $U$ is locally isomorphic to 
\begin{equation}\label{eq-quot}
\bC\times \bC^{n-1}/\frac{1}{m}(1, b_2,\dots, b_n)
\end{equation}
The natural projection $\pi: X'\rightarrow Y$ is induced by the map $(x_1, x_2, \dots, x_n)\mapsto (x_2, \dots, x_n)$. 

\begin{prop}
With the above notation, we have the formula:
\begin{equation}\label{eq-mld1}
    \mld(o, X)=\min_{p,g\neq 1}\left\{r, \frac{1}{m}\left(r w_1(g)+\sum_{i=2}^n w_i(g)\right) \right\}
\end{equation}
where $p$ ranges over all quotient singularities on $Y_0\subset X'$ and $g$ ranges over all non-identity elements in the stabilizer group $G_p\cong \bZ_m$ that satisfies $g^* x_i=e^{2\pi\sqrt{-1}w_i/m}x_i$ with $0\le w_i<m$. 
\end{prop}
We give a short sketch of the proof. 

\begin{proof}
The exceptional divisor of $\mu: X'\rightarrow X$ is equal to the zero section $Y_0$ of the orbifold line bundle $\CL^{-1}$ which is the underlying variety of the orbifold $\cY_0$. 
Moreover we have the formula $K_{X'}=\mu^* K_{X}+(r-1)Y_0$. 
So if we set $\mld(Y_0; X', (1-r)Y_0)=
\min\{A(v; X', (1-r)Y_0); \mathrm{center}(v)\subset Y_0\}$ where $v$ ranges over all divisorial valuations over $X'$, then there is an equality
$$
\mld(o, X)=\min\left\{r, \mld(Y_0;  X', (1-r)Y_0)\right\}. 
$$
 
Now we can use the formula for mld of quotient singularities (with a boundary divisor) as in \cite{LZ24}. Since cyclic quotient singularities are also locally toric, we can also use toric geometry to calculate the quantity as follows. Let $p\in Y_0$ be a point such that a neighborhood of $p$ is modeled on the quotient \eqref{eq-quot}. Let $\bZ^n$ be the standard lattice in $\bR^n$ and denote by $\Lambda'$ the larger lattice $\bZ^n+\bZ\frac{1}{m}(1, b_2, \dots, b_n)$. 
Let $\Lambda'\cap (0, 1]^n=\{v_1,\dots, v_{d_p}, \mathbf{1}\}$ with $\mathbf{1}=(1,1,\dots, 1)$. Then
$\mld(U\cap Y_0; X', (1-r)Y_0)=\min_{1\le k\le d_p}\{\sum_{i=2}^n x_i(v_k)+r x_1(v_k)\}$ where $x_j(v_k)$ denotes the $j$-th coordinate of the vector $v_k$ for $1\le j\le n$. On the other hand, each $v_k$ corresponds to $g_k\in G_p=\Lambda'/\bZ^n$ and $x_j(v_k)=\frac{1}{m}w_j(g_k)$. So we get the desired formula. 
\end{proof}

It is now convenient to introduce 
$\age(g)=\sum_{i=2}^n \frac{w_i(g)}{m}$, which is the $\age$ invariant in \cite{CR02} for the twisted sector corresponding to $g$ on the base orbifold $\cY$. We use $I\cY$ to denote the inertia orbifold of $\cY$ \cite[Definition 2.49]{Ruanbook}, on which $\age$ is locally constant \cite[Lemma 4.6]{Ruanbook}. Locally, connected components of $I\cY$ are indexed by the conjugacy classes of the isotropy groups. We use $\cY_{1}\subset I\cY$ to denote the non-twisted sector, which is diffeomorphic to $\cY$ and the inverse $g\mapsto g^{-1}$ makes sense on $\pi_0(I\cY)$. Then we can re-write \eqref{eq-mld1} in the following form:
\begin{equation*}
    \mld(o, X)=\min_{g\in \pi_0(I\cY),g\ne 1}\left\{r, \frac{rw_1(g)}{m}+\age(g) \right\}=\min_{g\in \pi_0(I\cY),g\ne 1}\left\{r, \frac{rw_1(g^{-1})}{m}+\age(g^{-1}) \right\}.
\end{equation*}
As $w_1(g^{-1})=0$ if $w_1(g)=0$ and $w_1(g^{-1})=m-w_1(g)$ if $w_1(g)\ne 0$, we have
\begin{equation}\label{eq-mld2}
    \mld(o, X)=\min_{g\in \pi_0(I\cY),g\ne 1}\left\{r, \frac{r(m-w_1(g))}{m}+\age(g^{-1}) \right\}=\min_{g\in \pi_0(I\cY)}\left\{\frac{r(m-w_1(g))}{m}+\age(g^{-1}) \right\}.
\end{equation}
The last equality follows from that $\frac{r(m-w_1(1))}{m}+\age(1^{-1})=r$.

Our main new observation in this note is that the quantity on the right-hand side of the above formula can be related to the dimension of the moduli space of twisted (or orbifold rational curves. This is motivated on the one hand by the fact that the lSFT invariant calculates the dimension of certain moduli space in symplectic field theory (see \S \ref{sec-symp} for an explanation) and on the other hand by similar formulae that appear in the study of orbifold Gromov-Witten invariants (\cite{CR02,AGV08})

Let $M$ be the smooth link of the Fano cone singularity $(X,o)$. We consider the finite set $S$ of isotropy groups (including the trivial group, which is the isotropy group for a generic point) of the $S^1$ action on $M$. The set $S$ is equipped with a partial order, we say $G_x>G_y$ if $G_y\subset G_x\subset S^1$ is a subgroup. For $G\in S$,  the quotient of the fixed point set $M^G/S^1$ gives rise to a branch locus $\cY_G$ of the quotient K\"ahler orbifold $\cY$ giving $M$ a stratification over the partial order set $S$. For non-minimal $G\in S$, we use $G^-$ to denote the unique maximal element that is smaller than $G$. We formally define $G^-=\emptyset$ when $G$ is the minimal element of $S$. By \cite[Proposition 3.3]{LZ24}, the Reeb orbits with period at most the period of a principle orbit (the simple orbit over a generic point) are parameterized by $M^G$ plus a multiplicity $k\in G\backslash G^-$. In particular, the space of unparameterized Reeb orbits with period at most the period of a principle orbit can be identified $I\cY$. 

We use $\frg=(g_1,\ldots,g_k)$ to denote a sequence of $k$ elements in $\pi_0(I\cY)$, which could be repetitive.
By an orbifold curve $\cC=(C, Q=\sum_{i=1}^k (1-\ell_i^{-1})q_i)$, we mean a smooth Riemann surface $C$ with an orbifold structure $\bC/\bZ_{\ell_i}$ with $\ell_i\in \bZ_{>0}$ at each $q_i\in \mathrm{Supp}(Q)$. 
Note that we allow the trivial orbifold structure at $q_i$ which corresponds to $\ell_i=1$. 
We will use the notion of twisted maps in the sense of Abramovich-Vistoli as defined in \cite{AV02}.\footnote{We always identify an orbifold with the corresponding Deligne-Mumford stack. } In particular, each marked point $q_i$ gives rise to a gerbe $\mathcal{Q}_i$ banded by $\bZ_{\ell_i}$ and we have an evaluation map $\ev_{q_i}(f)\in I\cY$.  
We denote by $\Mor_\frg((C,Q),\cY)$ the stack of twisted maps $f: \cC\rightarrow \cY$ that satisfies $\ev_{q_i}(f)\in \cY_{g_i}$ for each $1\le i\le k$ (see \cite{Ols06}). In the setting of differential geometry, this corresponds to the moduli space of representable good orbifold maps defined in \cite{CR02}, or equivalently representable orbifold maps using the groupoid languages in \cite[\S 5]{LU04}. For each $f\in \Mor_{\frg}((C, Q), \cY)$, we denote by 
$\Mor_{\frg}((C, Q), \cY; f|_Q)$ the subspace of twisted maps $\hat{f}$ in $\Mor_{\frg}((C, Q), \cY)$ that satisfy $\ev_{q_i}({\hat{f}})=\ev_{q_i}(f)$ for each $1\le i\le k$.

The following estimate is the orbifold version of the standard estimate in the smooth case (see \cite[2.11]{Deb01}). Note that on the right side of the estimate, we will use the intersection theory on Deligne-Mumford stacks as explained in \cite[\S 2]{AGV08}. 
\begin{lem}
With the above notation, we have the inequality
    \begin{equation}\label{eq-RR}
        \dim_{[f]}\Mor_{\frg}((C, Q), \cY; f|_Q)\ge -K_\cY\cdot f_*\cC+(1-g(C))\dim \cY-(\sum_{i=1}^k \age(g_i)+\dim \cY_{g_i}).
    \end{equation}   
\end{lem}
\begin{proof}
First, similar to the smooth case, by the deformation theory of the stacks (see \cite{Ols06}), we get
 \begin{eqnarray*}
\dim_{[f]}\Mor_\frg((C, Q), \cY; f|_Q)&\ge& \chi(\cC, f^*T\cY\otimes \otimes_i  \mathcal{I}_{\mathcal{Q}_i})
\end{eqnarray*}
    where $\mathcal{I}_{\mathcal{Q}_i}$ is the ideal sheaf of $\mathcal{Q}_i$. From the exact sequence, 
    \begin{equation*}
        0\rightarrow \mathcal{I}_{\mathcal{Q}_i}\rightarrow \mathcal{O}_{\cC}\rightarrow \cO_{\mathcal{Q}_i}\rightarrow 0
    \end{equation*}
   we get the identities:
    \begin{eqnarray*}
\chi(\cC, f^*T\cY\otimes \otimes_i  \mathcal{I}_{\mathcal{Q}_i})
&=&\chi(\cC, f^*T\cY)-\sum_i \chi(\mathcal{Q}_i, f^*T\cY)\\
&=&\chi(\cC, f^*T\cY)-\sum_i \chi(\mathcal{Q}_i, f^*T_{\cY_{g_i}})\\
&=&\left(-K_{\cY}\cdot f_*\cC+(1-g(C))\dim\cY-\sum_i \age(g_i)\right)-\sum_i \dim \cY_{g_i}. 
    \end{eqnarray*}
  For the second identity we used the tangent bundle lemma from \cite[Lemma 3.6.1]{AGV08}. The last identity used the Riemann-Roch theorem for twisted curves (\cite[Theorem 7.2.1]{AGV08}).
\end{proof}

When $C=\bP^1$ and $Q=0\cdot \{0\}+(1-\ell^{-1})\{\infty\}$, the orbifold curve $(\bP^1, Q)$ is identified with the weighted projective line $\bP^1(1, \ell)$ with $0=[1,0]$ and $\infty=[0,1]$. When $\mathfrak{g}=(1, g)$ we get from \eqref{eq-RR}:
    \begin{eqnarray*}
       \dim_{[f]}\Mor_{(1,g)}(\bP^1(1, \ell), \cY; f|_{\{0,\infty\}})&\ge& -K_{\cY}\cdot f_*\bP(1,\ell)+(\dim \cY-(\age(g)+\dim \cY_g))-\dim \cY\\
       &=&-K_{\cY}\cdot f_*\bP(1,\ell)+\age(g^{-1})-\dim \cY.
    \end{eqnarray*}
The last identity used the fact 
$\dim\cY-(\age(g)+\dim\cY_g)=\age(g^{-1})$ and is a main evidence that the formula \eqref{eq-mld2} is related to the dimension of moduli space of twisted rational curves. 

Now note that $f^*\CL^{-1}$ is an orbifold line bundle over $\bP^1(1, \ell)$ and, according to the definition of twisted curves, the morphism of the stabilizer groups $\bZ_\ell\rightarrow G_{f(\infty)}\cong \bZ_m$ is injective with $1$ mapped to $g$. By Setting $\la g\ra\cong \bZ_{\ell}$ to be a subgroup of $G_{f(\infty)}$, the twisted curve then satisfies $\ev(f)\in \cY_g$. Assume $f(\infty)$ is locally modeled on 
\begin{equation*}
    \bC\times \bC^{n-1}/\frac{1}{m}(1, b_2, \dots, b_n).
\end{equation*}
The generator of $\bZ_{\ell}$ is mapped to $p\in \bZ_m$ such that the order of $p$ is $\ell$. Note that $f^*\CL^{-1}=\cO_{\cC}(k)$ for $k\in \bZ_{<0}$ where $\cC=\bP(1,\ell)$. We have $\frac{p\ell}{m}=k \mod \ell$. As $w_1(g)=p \in [0,m)$, we have $k\le \frac{w_1(g)\ell}{m}-\ell$. Since $c_1(\cO_{\cC}(1))\cdot \cC=\frac{1}{\ell}$, we get the following inequality:
\begin{equation*}
    -K_{\cY}\cdot f_*\cC=r \CL\cdot f_*\cC =r f^*\CL\cdot \cC = r \cdot \frac{-k}{\ell} \ge r \cdot \frac{\ell-\frac{w_1(g)\ell}{m}}{\ell}=r\frac{m-w_1(g)}{m}. 
\end{equation*} 
For simplicity of notation, we set $d_{g^{-1}}(f)=-K_{\cY}\cdot f_*\cC+\age(g^{-1})$. Then the above discussion derives:
\begin{equation}\label{ineq-dg}
    \dim_{[f]}\Mor_{(1,g)}(\bP^1(1, \ell), \cY; f|_{\{0,\infty\}})+\dim \cY\ge d_{g^{-1}}(f)\ge \frac{r (m- w_1(g))}{m}+\age(g^{-1}). 
\end{equation}
Because the minimum \eqref{eq-mld2} is taken over all strata, we get
\begin{equation}\label{ineq-mld}
    \mld(o, X)\le \min \left\{ d_{g^{-1}}(f); f\in \Mor_{(1,g)}(\bP^1(1, \ell); f|_{\{0,\infty\}}), \ell\in \bZ_{>0}\right\}. 
\end{equation}

So in order to prove the main theorem, we just need to prove the inequality $d_{g^{-1}}(f)\le n=\dim \cY+1$ for some twisted map $f: \bP^1(1, \ell)\rightarrow \cY$. If $\cY$ is a smooth Fano manifold without orbifold points, then this was achieved by Mori via a reduction to characteristic $p>0$ and the bend-and-break method \cite{Mori79}. In \cite{CT09}, this has been partly generalized to the orbifold setting but the authors were mainly interested in the singular Fano varieties with special singularities. We will explain that Mori's method indeed gives us the sharp upper bound of mld in our problem. 

In the bend-and-break method, there are two steps. First, we need to obtain a non-trivial twisted map $f: \bP(1, \ell)\rightarrow \cY$. This is achieved by deforming a morphism from a smooth curve of genus $g\ge 1$. Let $C$ be a smooth Riemann surface of genus $g\ge 1$ and $f: C\rightarrow Y$ be a non-trivial morphism. First by \cite[Proof of Theorem 7.1.1]{AV02}, there exists a ramified cover $C'\rightarrow C$ such that $f$ lifts to a (twisted) morphism $f': C'\rightarrow \cY$ where $C'$ does not have non-trivial orbifold structure. 
\begin{prop}
    Let $f: C\rightarrow \cY$ be a morphism from a smooth curve of genus $g(C)\ge 1$. Fix any $c\in C$ 
    if $\dim_{[f]}\Mor(C, \cY; f|_{c})\ge 1$, then there exists a nontrivial twisted map $\hat{f}: \bP^1(1, \ell)\rightarrow \cY$. 
\end{prop}

\begin{proof}
First, we use the argument as in \cite[Proof of Corollary 1.7]{KM98} or \cite[Proof of Proposition 3.1]{Deb01}.
Let $T$ be the normalization of a 1-dimensional substack of $\Mor(C, \cY; f|_c)$ passing through $[f]$ and let $\bar{T}$ be a smooth compactification of $T$. By the same reasoning as in \cite[Proof of Corollary 1.7]{KM98} or \cite[Proof of Proposition 3.1]{Deb01}, there exists $t_0\in \bar{T}$ such that $\ev$ is not defined at $(c, t_0)$. The indeterminacies of the induced rational map $\ev: C\times \bar{T}\dashrightarrow Y$ can be resolved by blowing up points to get a morphism $\widetilde{\ev}: S\stackrel{\epsilon}{\rightarrow} C\times \bar{T}\stackrel{\ev}{\dashrightarrow} X$.  The fiber $F_{t_0}$ of $t_0$ under the projection $S\rightarrow \bar{T}$ is the union of the strict transform $\hat{C}$ of $C\times \{t_0\}$ and a connected exceptional rational 1-cycle $E$ which is not entirely contracted by $\widetilde{\ev}$. 
There is an irreducible $\bP^1$-component $E_1$ of $E$ such that $E_1$ intersects $\overline{F_{t_0}\setminus E_1}$ at a single node $p_0$.

Now by the proof of valuative criterion of twisted stable maps as in \cite[Proof of Proposition 6.0.1]{AV02} (see also \cite{CR02}), we can add orbifold structures to the nodes of the $F_{t_0}=\hat{C}\cup E$ so that the morphism $\widetilde{\ev}: F_{t_0}\rightarrow Y$ lifts to become a twisted map  $\widetilde{\ev}: \mathcal{F}_{t_0}\rightarrow \cY$. In more detail, at smooth points of $F_{t_0}$ we do not need to add a new orbifold structure. This is proved based on a base-change construction and a purity lemma (see \cite[Lemma 2.4.1]{AV02}). On a node $\{xy=0\}$ of $F_{t_0}$, there exists a chart $\{uv=0\}/\bZ_\ell$ where the action of $\bZ_\ell$ is described by $(u,v)\mapsto (\zeta u, \zeta^{-1}v)$ with $\zeta=\exp(2\pi\sqrt{-1}/\ell)$. In this chart, the map $\cF_{t_0}\rightarrow F_{t_0}$ is given by $x=u^\ell$ and $y=v^\ell$. 
Let $\mathcal{E}_1$ the orbifold curve $(E_1, (1-\ell^{-1})p_0)$.  
Then the induced twisted map $\hat{f}: \mathcal{E}_1\cong \bP^1(1, \ell)\rightarrow \cY$ is the wanted map. 
    
\end{proof}

In the second step, we can degenerate the twisted map $\bP^1$ to another one if the dimension of the deformation space is large. 
\begin{prop}\label{prop-bab1}
    Let $f: \cC=\bP^1(1, \ell)\rightarrow \cY$ be an twisted map such that $f(0)\in \cY_1$ and $f(\infty)\in \cY_g$. Assume that $$\dim_{[f]}\Mor_{(1,g)}(\bP^1(1,\ell), \cY; f|_{\{0,\infty\}})\ge 2,$$ then there exists a non-trivial twisted map $\hat{f}: \hat{\cC}=\bP^1(1, \hat{\ell})\rightarrow \cY$ such that $f_*[C]$ is numerically equivalent to $\hat{f}_*[\hat{\cC}]+\sum_i (f_i)_*[\cC_i]$ where $\{f_i: \cC_i\rightarrow \cY\}$ is a nonempty collection of twisted maps.  
\end{prop}
   Note that in general $\hat{f}(\infty)$ is different from $f(\infty)$. 
   \begin{proof}
  The proof is similar to that of Proposition \ref{prop-bab1}.  
       Let $T$ be the normalization of a 1-dimensional substack of $\Mor(\cC, \cY; f|_{\{0,\infty\}})$ passing through $[f]$ but not contained in its $\bC^*$-orbit and let $\bar{T}$ be a smooth compactification of $T$. By arguing as in \cite[Proposition 3.2]{Deb01} in the orbifold setting, there exists $t_0\in \bar{T}$ such that $\ev$ is not defined at $(0, t_0)$. We resolve the indeterminacy of the evaluation map $\ev: \bP^1\times \overline{T}\dashrightarrow X$ to get a fibered surface $\pi: S\rightarrow \bar{T}$ and a map $\tilde{\ev}: S\rightarrow X$ such that $F_{t_0}=\pi^{-1}(t_0)$ is a cycle of rational curves. There is an irreducible component $E_1$ of $F_{t_0}$ that intersects $\overline{F_{t_0}\setminus E_1}$ at a single node. 
       As before, we now use the same argument as \cite[Proof of Proposition 6.0.1]{AV02}
       to extend orbifold structures
       on $\bP^1\times T$ to $S$ (up to base change). This induces a $\bZ_{\hat{l}}$-orbifold structure at one point on $E_1\cong \bP^1$ and we get a twisted map $\hat{f}: \bP^1(1, \hat{\ell})\rightarrow \cY$ as wanted.  
   \end{proof}
  
We remark that the bend-and-break results in the orbifold setting similar to the above two propositions have also appeared in \cite[Proposition 3.5 and 3.6]{CT09}. 
Since the (image of) a twisted map can not break infinitely many times, the above two bend-and-break constructions produce a non-trivial bound (by also using \eqref{ineq-dg}).
\begin{prop}
    Assume that there exists a non-trivial smooth morphism $C \rightarrow \cY$ that satisfies 
    $\dim_{[f]}\Mor(C, \cY; f|_{c})\ge 1$. Then there exists a twisted map $f: \bP^1(1, \ell)\rightarrow \cY$ satisfying $f(0)\in \cY_1$, $f(\infty)\in \cY_g$ and  
     $$-K_{\cY}\cdot f_*\bP(1,\ell)+\age(g^{-1})-\dim \cY\le \dim_{[f]}\Mor_{(1,g)}(\bP^1(1, \ell), \cY; f|_{\{0,\infty\}})\le 1. $$
\end{prop}

To achieve the dimension condition in the above results, Mori introduced a groundbreaking method of reducing to the field of characteristic $p>0$ and using the Frobenius morphism to increase the dimension of the deformation space. He then applied the bend-and-break to get a rational curve in positive characteristics with a uniform upper bound on their degrees. Finally, Mori concluded the existence of a rational curve in characteristic zero by an algebraic argument (see \cite[pg.15]{KM98}, \cite[pg.62-63]{Deb01}). 
The same line arguments can be applied in the setting of orbifolds (see \cite{CT09}). Not that, in the characteristic $p$
argument, we need to assume that the orders of the stabilizers are relatively prime to $p$. In other words, we need to work with tamed Deligne-Mumford stacks in the sense of \cite{AV02}. This is not a restriction, since there are only finitely many strata and the prime $p$ can be arbitrarily large. 
As a consequence, we get \begin{prop}
   Let $\cY$ be a Fano orbifold. There exists a twisted map $f: \bP^1(1, \ell)\rightarrow \cY$ satisfying $f(0)\in \cY_1$, $f(\infty)\in \cY_g$ and  
     $$d_{g^{-1}}(f)=-K_\cY\cdot f_*\bP(1,\ell)+\age(g^{-1})\le \dim \cY+1. $$
\end{prop}

Combining this with the estimates \eqref{ineq-dg}-\eqref{ineq-mld}, we complete the proof of Theorem \ref{thm-main}. 

\begin{exmp}
    Consider the weighted projective line $\bP^1(2,3)=(\bC^2\setminus 0)/\bC^*$ where $\bC^*$ acts on $\bC^2$ by $\lambda\circ (z_1, z_2)=(\lambda^2 z_1, \lambda^3 z_2)$. 
    We have $I\cY=Y\sqcup \{([1,0]=0, -1)\} \sqcup \{([0,1]=\infty, \epsilon)\}\sqcup \{(\infty,\epsilon^{-1})\}$ with $\epsilon=e^{2\pi\sqrt{-1}/3}$.
    We consider the family of maps $f: \bC^*\times \bP^1(1,2)\rightarrow \cY$ given by
    $$f(t, [u_1,u_2])=f_t([u_1, u_2])= [t^2 u_1^2+u_2,  t^3 u_1^3+ u_1u_2].$$ 
    Then $\ev_0(f_t)=[t^2, t^3]=[1,1]$ and $\ev_\infty(f_t)= \{([1,0],-1)\}$ and $d_{-1}(f)=-K_{\cY}\cdot f_*C+\age(-1)=5\cdot  \frac{1}{2}+\frac{1}{2}=3>2= \dim \cY+1$. 
    Note that $f$ is undefined at $(0, [1,0])$. Indeed, in a chart near $0=[1,0]$, we set $v=u_2/u_1^2$ to get $f(t,v)=[t^2+v, t^3+v]$. We can now resolve the indeterminacy by a weighted blowup of weight $(1,3)$ to map $\hat{f}: \hat{\bC}^2:=Bl_{(1,3)}\bC^2\rightarrow Y$. Denote the exceptional divisor by $E_1\cong \bP^1(1,3)$, the strict transform of $\{t=0\}$ by $E_0$ and the strict transform of $\{v=0\}$ by $H$. Near $E_1\cap E_0$ we have a chart $\bC^2/\frac{1}{3}(1,-1)\mapsto \hat{\bC}^2$ given by $(x,w)\mapsto (xw, w^3)$ such that $E_0=\{x=0\}, E_1=\{w=0\}$
    and $$\hat{f}(x, w)=f(xw, w^3)=[x^2+w, x^3+1]. $$
    Note that when $x=0$, we get $\hat{f}(0,w)=[w,1]$ which can be seen as the restriction of identity self-map of $\bP^1(2,3)$.
    Near $E_1\cap H$ we have a chart given by $(t,y)\mapsto (t, t^3 y)$ such that $E_1=\{t=0\}$ and $$\hat{f}(t, y)=f(t, t^3y)=[1+ty, 1+y].$$ 
    Note that when $t=0$, we get $\hat{f}(0,y)=[1,1+y]$ which can be seen as the restriction of the map $\hat{f}: \cE=\bP^1(1,3)\rightarrow \bP^1(2,3)$ given by $[y_1, y_2]\rightarrow [1, 1+y_2/y_1^3]=[y_1^2, y_1^3+y_2]$. 
    
    So we indeed bend-and-break $f$ to a twisted map $f': \bP^1(1,3)\cup_{[0,1]_1=[0,1]_2} \bP^1(2,3)\rightarrow \bP^1(2,3)$ and by restriction get the above twisted map $\hat{f}: \cE=\bP^1(1,3)\rightarrow \bP^1(2,3)$ that satisfies $\ev_\infty(\hat{f})\in ([0,1],\epsilon)$ and $d_{\epsilon^{-1}}(\hat{f})=5\cdot \frac{1}{3}+\frac{1}{3}=2=\dim \cY+1$. 
\end{exmp}
\begin{exmp}
    Consider $\bP^1(2,3,5)=(\bC^3\setminus 0)/\bC^*$ where $\bC^*$ acts on $\bC^3$ by $\lambda\circ (z_1, z_2, z_3)=(\lambda^2 z_1, \lambda^3 z_2, \lambda^5 z_3)$. There is a family $\{f_t: \bP^1(1,3)\rightarrow \cY\}_{t\in \bC^*}$ of twisted map with fixed $f_t(0)$ and $f_t(\infty)$ given by: $$[u_1, u_2]\rightarrow [t^2 u_1^2,\; t^3 u_1^3+u_2,\; t^5 u_1^5+u_1^2 u_2].$$ We have $d_{\epsilon^{-1}}(f_t)=10\cdot \frac{1}{3}+\frac{1}{3}+\frac{1}{3}=4>3=\dim \cY+1$. Note that $f_t([1,0])$ is not defined at $t=0$ and one can bend-and-break $f_t$ to $f': \bP^1(1,5)\cup_{[0,1]_1=[0,1]_2} \bP^1(3,5)\rightarrow \bP^2(2,3,5)$ and get a new twisted map $\hat{f}: \bP^1(1,5)\rightarrow \bP^2(2,3,5)$ given by $[y_1, y_2]\mapsto [y_1^2, y_1^3, y_1^5+y_2]$ with $d_{e^{-2\pi\sqrt{-1}/5}}(\hat{f})=10\cdot \frac{1}{5}+\frac{3}{5}+\frac{2}{5}=3=\dim \cY+1$. 
\end{exmp}
Based on our discussion and examples above, we propose a conjectural characterization for weighted projective spaces as Fano orbifolds: 
\begin{conj}\label{conj-WP}
    Let $\cY$ be a Fano orbifold of dimension $n$. Assume that for any $\ell\in \bZ_{>0}$, $g\in \pi_0(I\cY)$ and any $f\in \Mor_{(1,g)}(\bP^1(1, \ell), \cY)$ we have $d_{g^{-1}}(f)=-K_{\cY}\cdot f_*\bP^1(1, \ell)+\age(g^{-1}) \ge n+1$.  Then $\cY$ is isomorphic to a finite quotient of a weighted projective space. 
\end{conj}
Note that this conjecture would imply that the equality in Theorem \ref{thm-main} holds only for smooth points. Indeed, the quotient orbifold $\cY=(X\backslash o)/\bC^*$ is a quotient of a weighted projective space if the equality in \Cref{thm-main} holds. As a consequence $(X,o)$ is an isolated quotient singularity, for which we know that the minimal discrepancy is equal to the dimension only when it is a smooth point. 

In the smooth case, the above conjecture is essentially the conjecture of Mori-Mukai as proved by Cho-Miyaoka-Shepherd-Barron (\cite{CMS02}, see also \cite{Keb02, Cheng17}). We expect that a careful implementation of the argument from \cite{CMS02, Keb02} in the orbifold setting will prove such a statement. We leave this to a future study and carry out such arguments here for Fano orbifolds with ample orbifold tangent bundles. 
\begin{proof}[Proof of Theorem \ref{thm-main2}]
    We sketch a proof that is modeled on Mori's proof in the case of smooth projective manifolds, which depends on the study of the tangent map of the rational curves at a fixed point. 
    Let $S$ be an irreducible component of $\Mor_{(1,g)}(\bP^1(1, \ell), \cY;f|_{\{\infty\}})$ that satisfies $$d_{g^{-1}}(f)=\chi(f^*T\cY)-\dim \cY_g=n+1,$$
    and $\Mor_{(1,g)}(\bP^1(1, \ell), \cY; f|_{\{\infty\}})/G$ is proper where $G$ is the automorphism of $\bP(1,\ell)$ fixing $\infty$.
    Since $\bP(1,\ell)$ has only one orbifold point, by \cite[Variation 2]{MT12}, we have a splitting
    \begin{equation}\label{eq-split}
        f^* T\cY=\cO(b_1)\oplus \cdots \oplus \cO(b_n). 
    \end{equation}
    Let $\{a_i\in [1,\ell]\}$ be the weights of $\bZ_\ell\hookrightarrow G_{f(\infty)}$ on $f^*T\cY|_\infty$ satisfying $a_1\le a_2\le \cdots \le a_n$. Then we have $a_i\equiv b_i \mod \ell $.  
    By the ampleness of $T\cY$, we have $b_i\ge 1$. By Bochner-Kodaira's vanishing, we have 
    \begin{equation}\label{eq-vob}
        H^1(\bP^1(1,\ell), \cO(b_i-k))=0, \text{ for } k<1+\ell+b_i.
    \end{equation}    
    By the Riemann-Roch for orbifold line bundles, we have
   $$\chi(\cO(b_i))=1+\frac{b_i}{\ell}-\frac{b_i\mod \ell}{\ell}\ge 1$$ 
   and the equality holds if and only if $b_i\in [1, \ell-1]$. 

   Assume $\dim \cY_g=d$. Then for $j\ge n-d+1$,  $a_j(g)=\ell$ and $b_j\equiv 0\mod \ell$. We have
   \begin{equation*}
       n+1=\chi(f^*T\cY)-\dim \cY_g=\sum_{j=1}^{n-d} (1+\frac{b_i-b_i\mod \ell}{\ell})+\sum_{j=n-d+1}^n \frac{b_j}{\ell}. 
   \end{equation*}
   Since each term is an integer at least 1, there must exist $i\in \{1,\dots, n\}$ such that:
    \begin{equation*}
        b_i=a_i+\ell, \quad \text{ and } \quad b_j=a_j \text{ for } j\neq i. 
    \end{equation*}
    In other words, we have a splitting 
    \begin{equation}\label{eqn:splitting}
        f^*T\cY=\cO(a_1)\oplus\cdots\oplus \cO(a_{i-1})\oplus \cO(a_{i}+\ell)\oplus \cO(a_{i+1})\cdots \oplus \cO(a_n).
    \end{equation}
    The differential $df$ corresponds to a section of 
    \begin{equation}\label{eqn:tangent_splitting}
        \cO(-(1+\ell))\otimes f^*T\cY
    =\cO(a_1-\ell-1)\oplus\cdots\oplus \cO(a_{i-1}-\ell -1)\oplus \cO(a_{i}-1)\oplus \cO(a_{i+1}-\ell-1)\cdots \oplus \cO(a_n-\ell-1)
    \end{equation}
    which is equivalently a section of $H^0(\bP^1(1,\ell), a_i-1)\cong \bC$ since $\cO(a_j-\ell-1)$ with $a_j-\ell-1<0$ has no holomorphic sections. 
    Locally $f:\bC/\frac{1}{\ell}(1) \to \bC^n/G_{f(\infty)}$ factors through $\bC^n/\bZ_{\ell}$ and has a lifting to $\bC\rightarrow \bC^n$ whose leading order terms are given as:
    \begin{equation*}
        f(z)\sim (t_1 z^{a_1}+t'_1 z^{a_1+\ell}, \dots, t_i z^{a_i}+t'_i z^{a_i+\ell}, \dots, t_n z^{a_n}+t'_nz^{a_n+\ell}). 
    \end{equation*}
    Then $(t_1,\dots, t_n)\ne 0$, for otherwise, $\rd f$ will have a zero of order $\ge \ell$ in a component of the splitting \eqref{eqn:tangent_splitting}, contradicting with the splitting. Then we can define a tangent map:
    \begin{equation*}
        \tilde{\phi}(f)=(t_1,\dots, t_n)\in \bC^n/\frac{1}{\ell}(a_1,\dots, a_n)
    \end{equation*}
    and the corresponding projective tangent map $\phi: S\rightarrow \bP_{\mathbf{w}}$ where $\mathbf{w}=(a_1, \dots, a_n)$. 
    Since the infinitesimal variation of $f$ with a fixed $\tilde\phi(f)=(t)$ correspond to variation of $t'_i$, the fiber of $\tilde{\phi}$ over $[c]\neq 0$ has an orbifold tangent space and an obstruction space given by $H^j(\bP^1(1,\ell), \cO\oplus \cO(-\ell)^{n-1})$ which is equal to $\bC$ for $j=0$ and vanishes for $j=1$. 
    
    The automorphism group $G$ of $\bP(1,\ell)$ fixing $\infty$ is given by $\{[z,w]\mapsto [az,bz^\ell+cw]|a,c\ne 0 \}$ where $\infty=[0,1]$. The group $G$ acts on $S\times \bP^1(1, \ell)$ and we take the quotient $\cM=S/G$ as an orbifold. The tangent space of $S$ at $f$ is given by  $H^0(\bP^1(1,\ell), \cO^{n-1}\oplus \cO(\ell) )\cong \bC^{n+1}$ and tangent of the $G$-action is $H^0(\bP(1,\ell),\cO(\ell))$, hence the tangent space of $\cM$ at $[f]$ is  $H^0(\bP^1(1,\ell), \cO^{n-1})\cong \bC^{n-1}$. It is straightforward to check that the induced smooth orbifold morphism
    $\bar{\phi}: \cM\rightarrow \bP_{\mathbf{w}}$ is finite of maximal rank and inject  the stabilizer of $[f]$ into the  the stabilizer of $\bar{\phi}([f])$, i.e.\ $\bar{\phi}$ is representable. 
    In particular, $\bar{\phi}$ is an orbifold covering in the sense of \cite[Definition 2.16]{Ruanbook}, therefore we conclude that $\cM$ is a weighted projective space of weight $\bar{\mathbf{w}}=\mathbf{w}/k$ where $k$ is a divisor of $\mathrm{gcd}(\bar{\mathbf{w}}):=\mathrm{gcd}(a_1, \dots, a_n)$ such that $\mathrm{gcd}(\bar{\mathbf{w}})=\frac{\mathrm{gcd}(\mathbf{w})}{k}$ is equal to the order of the stabilizer at the generic point of $\cM$. 
    
    Note that locally near $f(\infty)$, $\cY\cong \bC^n/G_{f(\infty)}$ 
    is a quotient of $\cY':=\bC^n/\frac{1}{\ell}(a_1, \dots, a_n)$. Let $\mu': \hat{\cY}'\rightarrow \cY'$ denote the weighted blowup of $\cY'$ at $[0]$ with the weight $\mathbf{w}$ and $\mu: \hat{\cY}\rightarrow \cY$ denote the weighted blow-up at $f(\infty)$ with weight $\mathbf{w}$. Denote by $E'$ and $E$ the exceptional divisors of $\mu'$ and $\mu$ respectively. Then $E'\cong \bP_{\mathbf{w}}$ with $\cO(E')|_{E'}=\cO_{\bP_{\mathbf{w}}}(-\ell)$.    
    The exceptional divisor $E$ is thus a finite quotient of weighted projective space $\bP_{\textbf{w}}$ and  the orbifold line bundle $\cO(E)|_E$ is the quotient of $\cO_{\bP_\mathbf{w}}(-\ell)$. 

    Consider the evaluation map $\Phi: S\times \bP^1(1, \ell)\rightarrow \cY$ such that $\Phi|_{\{s\}\times \bP(1, \ell)}$ is the morphism represented by $s$.  We have an induced orbifold fibration $\cZ=(S\times \bP^1(1,\ell))/G\rightarrow \cM\cong \bP_{\bar{\mathbf{w}}}$ which is an orbifold $\bP(1,\ell)$ fibration with a section $\sigma: \cM\rightarrow \cZ$ by $\sigma([f])=[(f,\infty)]$. The fiber of $\Phi|_{S\times (\bP(1,\ell)\backslash\{\infty\})}$, at $(f,q)$, has an orbifold tangent space $H^0(\bP(1,\ell),\cO(-\ell)^{n-1}\oplus \cO)\oplus \bC$ and the obstruction vanishes, where the $\bC$ summand is from the tangent of $q\in  \bP(1,\ell)\backslash \{0\}$ and the first component of is the subspace of the tangent space of $S\simeq H^0(\bP(1,\ell),\cO^{n-1}\oplus \cO(\ell))$ at $f$ subject to the condition that it vanishes at the smooth point $q\in \bP(1,\ell)$. As the tangent space of the fiber is the tangent of the $G$ action, the induced morphism $\bar{\Phi}: \cZ\to \cY$ is of maximal rank restricted to $\cZ\backslash\sigma(\cM)$. We claim the orbifold morphism $\bar{\Phi}|_{\cZ\backslash \sigma(\cM)}$ is also representable. If $g\in G$ fixes $(f,0)$, then on $\bC/\frac{1}{\ell}(1)$ neighborhood around $\infty \in \bP(1,\ell)$, $g$ is written as $z\mapsto \lambda z$ such that $\lambda \sim \lambda e^{2\pi \sqrt{-1}/\ell}$ if $f\circ g = f$ near $0$ in $\bC/\frac{1}{\ell}(1)$. Hence the stabilizer $G_{[(f,0)]}$ is a cyclic group. If $G_{[(f,0)]}$ does not inject into the stabilizer of $f(0)$ in $\cY$ with a kernel $\bZ_s$, then $f$ is a branch cover over a twisted map $f':\bP(1,\ell')\to \cY$ of order $s$, where $\ell$ divides $\ell'$ and $\ell'/\ell$ divides $s$. As a consequence, we have 
    $$(f')^*T\cY=\cO\left(\frac{a_1\ell'}{s\ell}\right)\oplus \dots \oplus \cO\left(\frac{a_{i-1}\ell'}{s\ell}\right)\oplus \cO\left(\frac{a_{i}\ell'}{s\ell}+\frac{\ell'}{s}\right)\oplus  \cO\left(\frac{a_{i+1}\ell'}{s\ell}\right)\oplus \dots \oplus  \cO\left(\frac{a_{n}\ell'}{s\ell}\right),$$
    over $\bP(1,\ell')$. All the weights above are no larger than $2\ell'/s\le \ell'$. Since $\chi((f')^*T\cY)-\dim \cY_{g'}\ge n+1$ where $f'(\infty)\in \cY_{g'}$, we can argue as in \eqref{eqn:splitting} that at least one weight is larger than $\ell'$, contradiction.

    The evaluation map $\widetilde{\Phi}: S\times \bP^1(1, \ell)\dashrightarrow \hat \cY$ is well-defined near $S\times \{\infty\}$ and the induced map $\bar{\widetilde{\Phi}}:\cZ\dashrightarrow \hat \cY$ near $\sigma(\cM)$ is a covering map. Therefore the orbifold normal bundle of $\sigma(\cM)$ is isomorphic to $\cO_{\bP_{\bar{\mathbf{w}}}}(-\ell)$. As a consequence, we have $\cZ=\bP(\cO_{\bP_{\bar{\mathbf{w}}}}(-\ell)\oplus \cO)$. Then by blowing down $\sigma(\cM)$ in $\bar{\Phi}$, we get an orbifold morphism $\bP(\bar{\mathbf{w}},\ell)\to \cY$ that is smooth and representable, and hence an orbifold covering. 

    \begin{equation*}
    \xymatrix 
   {
    & S\times \{\infty\} \ar[ldd]_{\phi} \ar@{^{(}->}[r] \ar[d]^{\bullet/G} & S\times \bP^1(1,\ell) \ar[d]^{\bullet/G} & \\
    &\bP_{\bar{\mathbf{w}}}\cong \sigma(\cM) \ar[dl]^{\bar{\phi}} \ar[d]^{} \ar@{^{(}->}[r] & \cZ \ar[dr]^{\bar{\Phi}}\ar[r] \ar@{-->}[d]^{\bar{\widetilde{\Phi}}}&  \bP(\bar{\mathbf{w}}, \ell) \ar[d] \\
    \bP_{\mathbf{w}}\cong E' \ar[r] & E \ar@{^{(}->}[r] & \hat{\cY} \ar[r]^{\mu} & \cY 
   }
\end{equation*}
\end{proof}

\section{Relation with moduli space in symplectic field theory}\label{sec-symp}
In \cite[Proposition 3.4]{LZ24}, we also derived the mld in terms of the lSFT invariant from symplectic geometry
\begin{equation}\label{eqn:mld_SFT}
    \mld(o, X)=\frac{1}{2}\inf_{\gamma}\lSFT(\gamma)+1
\end{equation}
where $\gamma$ ranges over all Reeb orbits of the link of $x\in X$ with respect to any conic contact form. We refer to \cite[\S 2.3]{LZ24} for Conley-Zehnder indices and SFT degrees.

Let $\frg=(g_1,\ldots, g_k)$ be a sequence of elements in $\pi_0(I\cY)$. We define $\moduli_{0,k}(\cY, A,\frg)$ to be the compactified moduli spaces of orbifold $\bP^1$ curves with $k$ marked points (only marked points can be orbifold points) to $\cY$ with homology class $A$, such that $i$th marked point is mapped to $\cY_{g_i}$. The order of the $i$th marked point is the order of $g_i$, hence the associated orbifold map $(\bP^1,Q)\to \cY$ is representable as in \cite{AGV08}. Then the complex virtual dimension of  $\moduli_{0,k}(\cY, A,\frg)$ is given by \cite[Theorem A]{CR02}
$$\langle c_1(T\cY),A\rangle+\dim_{\bC}(\cY)+k-3-\sum_{i=1}^k \age(g_i).$$
This is equivalent to the right-hand-side of \eqref{eq-RR} with $k-3$ from the freedom of marked points and automorphisms and without $\sum_{i=1}^k\dim_{\bC}\cY_{g_i}$ as we do not require marked points going through fixed points in $\cY_{g_i}$. As in \S \ref{s2}, we will be mainly interested in the case with one marked point. We use $\moduli^{\bullet}_{0,1}(\cY,A,g)$ to denote the compactified moduli space of orbifold $\bP^1(1,\ell)$ in $\cY$ of homology class $A$ with the order $\ell$ marked point mapped to a chosen point on $\cY_g$, where the order of $g$ is $\ell$. This is the compactification of the Deligne–Mumford stack $\Mor(\bP^1(1,\ell),\cY;f|_\infty)/\Aut(\bC)$. The virtual dimension is
$$\langle c_1(T\cY),A\rangle+\dim_{\bC}\cY-2-\age(g)-\dim_{\bC}\cY_g,$$
which is smaller than \eqref{ineq-dg} by $2$ from the automorphism of $\bP(1,\ell)$.

We have a correspondence between $\pi_0(I\cY)$ with the connected components of the Morse-Bott family or Reeb orbits with period at most the principle orbit. We may rescale the contact form, such that the principle orbit (the simple Reeb orbit over a non-singular point of $M/S^1$) has period $1$. For $g\in \pi_0(I\cY)$, we use $\gamma_g$ to represent the corresponding Reeb orbit. We use $\gamma_{g,l}$ to denote the Reeb orbit $\gamma_g$ followed by $l\in \mathbb{N}$ principle orbits. They form all the Reeb orbits on $M$. Let $W$ be a strong symplectic (orbifold) filling of $M$, we use $\moduli_{\SFT}(W,A,\gamma_{g,l})$ to denote the compactified moduli spaces of holomorphic planes in the completion $\widehat{W}$ of homology class $A$ and one positive puncture asymptotic to $\gamma_{g,l}$, see e.g.\cite[\S 1.5]{SFT}. The \emph{real} virtual dimension of $\moduli_{\SFT}(W,A,\gamma_{g,l})$ is given by \cite[Proposition 1.7.1
]{SFT} (strictly speaking, we are considering the Morse-Bott case with the holomorphic curve asymptotic to a fixed orbit in the family, therefore we use $\mu_{\LCZ}$)
\begin{equation}\label{eqn:SFT}
    \mu_{\LCZ}^A(\gamma_{g,l})+\dim_{\bC}W-3,
\end{equation}
where $\mu_{\LCZ}^A(\gamma_{g,l})$ is the lower semi-continuous Conley-Zehnder index of $\gamma_{g,l}$ using trivializations of $\det_{\bC}u^*TW$ such that $u:\bD\to W$ is a continuous (orbifold) map with boundary mapped to $\gamma_{g,l}$ with homology class $A$. In our case, the unit disk bundle $D(\CL^{-1})$ in $\CL^{-1}$ is naturally a symplectic orbifold filling of $M$ and we have the following.

\begin{prop}\label{prop:SFT}
    Let $\gamma$ be a Reeb orbit and $u$ a disk in $D(\CL^{-1})$ with boundary $\gamma_{g,l}$, whose intersection number with the zero section is zero. Then we have $\lSFT(\gamma_{g,l}) = \vdim_{\bR} \moduli_{\SFT}(D(\CL^{-1}),[u],\gamma_{g,l})$.
\end{prop}
\begin{proof}
    By definition $\lSFT(\gamma_{g,l})=\mu^{\bQ}_{\LCZ}(\gamma_{g,l})+\dim_{\bC} D(\CL^{-1})-3$, where $\mu^{\bQ}_{\LCZ}(\gamma_{g,l})$ is computed using a trivialization of $\det_{\bC} \oplus^N \xi$ for the contact structure $\xi$ of $M$ and some $N\in \bZ_{>0}$. By \cite[Lemma 3.1]{LZ24}, as $c_1(D(\CL^{-1}))$ viewed in $H^2(D(\CL^{-1}),M;\bQ)$ is Lefschetz dual to a multiple of the fundamental class of the zero section, we have $\la c_1(D(\CL^{-1})), [u]\ra=0$ and $\lSFT(\gamma_{g,l}) = \mu^u_{\LCZ}(\gamma_{g,l})+\dim_{\bC} D(\CL^{-1})-3 = \vdim_{\bR} \moduli_{\SFT}(D(\CL^{-1}),[u],\gamma_{g,l})$.
\end{proof}

We restate \eqref{ineq-mld} as follows and reprove it using \eqref{eqn:mld_SFT} as the definition of mld. The proof is disguised as degenerating the SFT curves in $D(\CL^{-1})$ into fiber holomorphic disks and holomorphic curves contained in the zero section $\cY_0$, hence we can relate their virtual dimensions. However, as we only care about the virtual dimension, we do not require such degeneration actually happens. One can verify the formulae in the proof directly using the virtual dimension formulae without appealing to the degeneration picture. 
\begin{prop}\label{prop:md_vd}
    Let $(X,o)$ be a $n$-dimensional Fano cone singularity over Fano orbifold $\cY$ with associated ample line bundle $\CL$. If $\moduli^{\bullet}_{0,1}(\cY, A, g)\ne \emptyset$, we have 
    $$\mld(o,X) \le \vdim_{\bC} \moduli^{\bullet}_{0,1}(\cY, A, g)+2.$$
\end{prop}
\begin{proof}
    We start with the seemingly simpler case where $g=1$ represents the identity component $\cY$ of $I\cY$, i.e.\ the marked point on $\bP^1$ is not an orbifold point. Assume $u \in \moduli^{\bullet}_{0,1}(D(\CL^{-1}),A,1)$. Here we treat $u$ as irreducible, the nodal case is the same. Then $u^*\CL$ is an ample bundle over $\bP^1$. We write $k=\la c_1(\CL),A\ra>0$, let $v$ be the fiber disk in $D(\CL^{-1})$ over a smooth point in $\cY$ and $v^k$ be the natural $k$th branched cover of $v$. Then we have $\vdim_{\bR} \moduli_{\SFT}(D(\CL^{-1},[v^k], \gamma^k_{1})=2k-2$. This follows from that $\mu_{LCZ}^{v^k}(\gamma^k_{1})=2k-\dim_{\bC}\cY$ as the linearized flow of $\gamma^k_{1}$ using the induced trivialization from $v^k$ is rotation by $k$ rounds in the fiber direction and identity on the complement. Then the claim follows from \eqref{eqn:SFT}. Since $v^k\#u$ has intersection number zero with $\cY_0$, by \Cref{prop:SFT}, we have 
    \begin{eqnarray}
    \lSFT(\gamma^k_\Id) &= &\vdim_{\bR}\moduli_{\SFT}(D(\CL^{-1}), v^k\#u, \gamma^k_{1}) \nonumber \\
    &=& \vdim_{\bR}\moduli_{\SFT}(D(\CL^{-1}), [v^k], \gamma^k_{1})+4+\vdim_{\bR}\moduli^{\bullet}_{0,1}(D(\CL^{-1}), [u], 1). \label{eqn:fiberproduct}
    \end{eqnarray}
    To see the last equality, the nodal curve $v^k\cup u$ is contained in $\moduli_{\SFT}(D(\CL^{-1}), v^k\#u, \gamma^k_{1})$, those nodal curves form a strata of $\moduli_{\SFT}(D(\CL^{-1}), v^k\#u, \gamma^k_{1})$ of virtual codimension $2$. On the other hand, those nodal curves form the fiber product $\moduli_{\SFT,1}(D(\CL^{-1}), [v^k], \gamma^k_{1})\times_{D(\CL^{-1})}\moduli_{0,1}(D(\CL^{-1}), [u], 1)$, where $\moduli_{\SFT,1}(D(\CL^{-1}), [v^k], \gamma^k_{1})$ is the SFT curves as in $\moduli_{\SFT}(D(\CL^{-1}), [v^k], \gamma^k_{1})$ but with one marked point and the fiber product is taken over the evaluation maps on the marked points. Then we have $\vdim_{\bR} \moduli_{\SFT,1}(D(\CL^{-1}), [v^k], \gamma^k_{1})\times_{D(\CL^{-1})}\moduli_{0,1}(D(\CL^{-1}), [u], 1)$ is
    \begin{eqnarray*}
         & & \vdim_{\bR} \moduli_{\SFT,1}(D(\CL^{-1}), [v^k], \gamma^k_{1})+\vdim_{\bR} \moduli_{0,1}(D(\CL^{-1}), [u], 1)-\dim_{\bR}D(\CL^{-1})\\
         & = & \vdim_{\bR} \moduli_{\SFT}(D(\CL^{-1}), [v^k], \gamma^k_{1})+2+\vdim_{\bR}\moduli^{\bullet}_{0,1}(D(\CL^{-1}), [u], 1).
    \end{eqnarray*}
    Therefore \eqref{eqn:fiberproduct} holds. Note that
    $$\vdim_{\bR}\moduli_{0,1}^{\bullet}(D(\CL^{-1}),[u],1) = 2+\vdim_{\bR}\moduli_{0,1}^*(\cY,[u],1) - 2\la c_1(\CL),A \ra -2.$$
    The first $2$ is from that the point constraint in $\moduli_{0,1}^{\bullet}(D(\CL^{-1}),[u],1)$ is of real codimension $\dim_{\bR}D(\CL^{-1})$, while the point constraint in $\moduli_{0,1}^{\bullet}(\cY,[u],1)$ is of real codimension $\dim_{\bR}\cY$. The $- 2\la c_1(\CL),A \ra -2$ comes from the index of Cauchy-Riemann operator in the normal direction of $\cY_0\subset D(\CL^{-1})$. Combining them together, we have
    $$\lSFT(\gamma^k_{1}) = \vdim_{\bR}\moduli_{0,1}^{\bullet}(\cY,[u],1)+2$$
     In general, assume $g$ is represented by $G\in S$ and multiplicity $k\ne 0 \in G\backslash G^-\subset \bZ/|G|$. The orbifold marked point on $\bP^1$ then has order $|G|/\gcd(k,|G|)$. Then $u^*(\CL^{-1})$ has a meromorphic section without zero and an order $-\la c_1(\CL^{-1}),A\ra$ pole at the marked point. Then $-k/|G|+\la c_1(\CL^{-1}),A\ra$ is a negative integer $-l$. We consider the Reeb orbit $\gamma_{g,l-1}$. Let $v$ be the fiber orbifold disk over a generic point in $M^G/S^1$, with boundary $\gamma_{g,l-1}$ and an orbifold marked point asymptotic to $\cY_g\subset ID(\CL^{-1})$, we use $\vdim_{\bR}\moduli_{\SFT,1}(D(\CL^{-1}),[v], \gamma_{g,l-1},g)$ to denote the moduli space of such $v$, then 
     $$\vdim_{\bR}\moduli_{\SFT,1}(D(\CL^{-1}),[v], \gamma_{g,l-1},g)=2l-2.$$
     as this can be computed in the local model $\bC\times \bC^{n-1}/\frac{1}{|G|}(1,b_2,\ldots,b_n)$.
     Since we have 
     \begin{eqnarray*}
         \vdim_{\bR}\moduli_{0,1}^{\bullet}(D(\CL^{-1}),[u], g) & = & \vdim_{\bR}\moduli_{0,1}^{\bullet}(\cY,[u], g)+2+2\la c_1(\CL^{-1}),[u]\ra-2k/|G|\\
         & = & \vdim_{\bR}\moduli_{0,1}^{\bullet}(\cY,[u], g)+2-2l.
     \end{eqnarray*}
     where $2+2\la c_1(\CL^{-1}),[u]\ra-2k/|G|$ is the Fredholm index of the Cauchy-Riemann operator in the normal direction. 
     Since $v\#u$ has zero intersection with $\cY_0$, by \Cref{prop:SFT}, we have
     \begin{eqnarray*}
     \lSFT(\gamma_{g,l-1}) & = &\vdim_{\bR}\moduli_{\SFT}(D(\CL^{-1}), [v\# u],\gamma_{g,l-1})\\
     & = & 2+ \vdim_{\bR}\moduli_{\SFT,1}(D(\CL^{-1}), [v],\gamma_{g,l-1})\times_{\cY_g} \moduli_{0,1}(D(\CL^{-1}),[u],g)\\
     & = &  2+ \vdim_{\bR}\moduli_{\SFT,1}(D(\CL^{-1}), [v],\gamma_{g,l-1})+\vdim \moduli_{0,1}(D(\CL^{-1}),[u],g) -\dim_{\bR}\cY_g\\
     & = &  2+ \vdim_{\bR}\moduli_{\SFT,1}(D(\CL^{-1}), [v],\gamma_{g,l-1})+\vdim_{\bR}\moduli_{0,1}^{\bullet}(D(\CL^{-1}),[u],g).
     \end{eqnarray*}
     Then $\lSFT(\gamma_{g,l-1})=2+\vdim_{\bR}\moduli_{0,1}^{\bullet}(\cY,[u],g)$. Since $2\mld(X,o)\le \lSFT(\gamma)+2$, we have 
     $$\mld(X,o) \le \vdim_{\bC} \moduli_{0,1}^{\bullet}(\cY,[u], g)+2.$$
     
\end{proof}

\bibliographystyle{alpha} 
\bibliography{ref}

\begin{thebibliography}{CMSB02}

\bibitem[AGV08]{AGV08}
Dan Abramovich, Tom Graber, and Angelo Vistoli.
\newblock Gromov–witten theory of deligne–mumford stacks.
\newblock {\em Amer. J. Math.}, 130:1337–1398, 2008.

\bibitem[ALR07]{Ruanbook}
Alejandro Adem, Johann Leida, and Yongbin Ruan.
\newblock {\em Orbifolds and stringy topology}, volume 171 of {\em Camb. Tracts Math.}
\newblock Cambridge: Cambridge University Press, 2007.

\bibitem[Amb06]{Amb06}
Florin Ambro.
\newblock The minimal log discrepancy.
\newblock {\em arXiv preprint arXiv:0611859}, 2006.

\bibitem[AV02]{AV02}
Dan Abramovich and Angelo Vistoli.
\newblock Compactifying the space of stable maps.
\newblock {\em J. Amer. Math. Soc.}, 15(1):27--75, 2002.

\bibitem[Che17]{Cheng17}
Jiun-Cheng Chen.
\newblock Characterizing projective spaces for varieties with at most quotient singularities.
\newblock {\em Bull. Inst. Math. Acad. Sin. (N.S.)}, 12(4):297--314, 2017.

\bibitem[CMSB02]{CMS02}
Koji Cho, Yoichi Miyaoka, and N.I. Shepherd-Barron.
\newblock Characterizations of projective space and applications to complex symplectic manifolds.
\newblock In {\em Heigher dimensional birational geometry (Kyoto, 1997), Adv. Stud. Pure Math., 25}, pages 1--88. Mathematical Society of Japan, 2002.

\bibitem[CR02]{CR02}
Weimin Chen and Yongbin Ruan.
\newblock Orbifold gromov-witten theory.
\newblock In {\em Orbifolds in Mathematics and Physics}, volume 310 of {\em Contemp. Math.}, page 25–85. Amer. Math. Soc., Providence, RI, 2002.

\bibitem[CT09]{CT09}
Jiun-Cheng Chen and Hsian-Hua Tseng.
\newblock Cone theorem via {Deligne}-{Mumford} stacks.
\newblock {\em Math. Ann.}, 345(3):525--545, 2009.

\bibitem[Deb01]{Deb01}
Olivier Debarre.
\newblock {\em Higher-dimensional algebraic geometry}.
\newblock Springer-Verlag, New York, 2001.

\bibitem[EGH00]{SFT}
Y.~Eliashberg, A.~Givental, and H.~Hofer.
\newblock Introduction to {Symplectic} {Field} {Theory}.
\newblock In {\em GAFA 2000. Visions in mathematics---Towards 2000. Proceedings of a meeting, Tel Aviv, Israel, August 25--September 3, 1999. Part II.}, pages 560--673. Basel: Birkh{\"a}user, 2000.

\bibitem[Keb02]{Keb02}
Stephan Kebekus.
\newblock Characterizing the projective space after cho, miyaoka and shepherd-barron.
\newblock In {\em Complex Geometry (G\"{o}ttingen, 2000)}, pages 147--155. Springer-Verlag, Berlin, 2002.

\bibitem[KM98]{KM98}
Janos Kollar and Shigefumi Mori.
\newblock {\em Birational Geometry of Algebraic Varieties}.
\newblock Cambridge Univ. Press, Cambridge, 1998.

\bibitem[LU04]{LU04}
Ernesto Lupercio and Bernardo Uribe.
\newblock Gerbes over orbifolds and twisted k-theory.
\newblock {\em Commun. Math. Phys.}, 245:449--489, 2004.

\bibitem[LZ24]{LZ24}
Chi Li and Zhengyi Zhou.
\newblock Kähler compactification of $\mathbb{C}^n$ and reeb dynamics.
\newblock {\em arXiv preprint arXiv:2409.10275}, 2024.

\bibitem[McL16]{McL16}
Mark McLean.
\newblock Reeb orbits and the minimal discrepancy of an isolated singularity.
\newblock {\em Invent. Math.}, 204(2):505--594, 2016.

\bibitem[Mor79]{Mori79}
Shigefumi Mori.
\newblock Projective manifolds with ample tangent bundles.
\newblock {\em Ann. of Math.}, (110):593--606, 1979.

\bibitem[MT12]{MT12}
Johan Martens and Michael Thaddeus.
\newblock Variations on a theme of grothendieck.
\newblock {\em arXiv preprint arXiv:1210.8161}, 2012.

\bibitem[Ols06]{Ols06}
Martin Olsson.
\newblock Hom-stacks and restriction scalars.
\newblock {\em Duke Math. J.}, 134(1):139--164, 2006.

\end{thebibliography}

\Addresses
\end{document}